\documentclass[a4paper, 10pt,]{article}

\usepackage[utf8]{inputenc}
\usepackage[english]{babel}
\usepackage[ntheorem]{empheq}
\usepackage{graphicx} 
\usepackage{tikz}
\usepackage{pgfplots}

\usepackage{amsthm, amsfonts, amssymb, amscd}
\usepackage{caption}
\usepackage{subcaption}
\usepackage{mdframed}
\usepackage{lipsum}
\usepackage{tikz}
\usepackage{enumerate}
\usepackage{makeidx}
\usetikzlibrary{arrows}
\usetikzlibrary{calc}
\usepackage{graphicx}
\usepackage[hyphens]{url}

\def\R{\mathbb{R}}
\def\V{\mathbb{V}}

\def\Q{\overline{Q}}
\def\u{\overline{u}}
\def\K{\mathcal{K}}
\def\F{\mathcal{F}}
\def\C{\mathcal{C}}

\def\jj{\mathrm{j}}

\def\C{C_T}
\DeclareMathOperator{\dls}{d_{LS}}
\DeclareMathOperator{\db}{d_{B}}

\def\QQ{\underline{Q}}
\def\uu{\underline{u}}

\newtheorem{theorem}{Theorem}
\newtheorem{proposition}{Proposition}
\newtheorem{lemma}{Lemma}
\newtheorem{remark}{Remark}
\newtheorem{definition}{Definition}
\newtheorem{corollary}{Corollary}[theorem]

\makeatletter
\newcommand{\Hy}[1]{#1\renewcommand{\@currentlabel}{#1}}
\makeatother

\newcommand\norm[1]{\|#1\|}
\newcommand\restr[2]{{
  \left.\kern-\nulldelimiterspace 
  #1 
  \vphantom{\big|} 
  \right|_{#2} 
  }}

\title{Periodic solutions for a nonautonomous mathematical model of hematopoietic stem cell dynamics}

\author{Mostafa Adimy$^{a}$, Pablo Amster$^{b}$ and  Juli\'an Epstein$^{b}$}

\date{\today}

\begin{document}
	
	\maketitle
	
	\begin{center}
		$^{a}$ Inria, Univ Lyon, Universit\'{e} Lyon 1, CNRS UMR 5208, Institut Camille Jordan, 43 Bd. du 11 novembre 1918, F-69200 - Villeurbanne Cedex, France
		\vspace{0.2cm}
		
		$^{b}$ Departamento de Matem\'atica, Facultad de Ciencias Exactas y Naturales, Universidad de Buenos Aires \& IMAS-CONICET\\ 
		Ciudad Universitaria - Pabell\'on I, 1428, Buenos Aires, Argentina
		\vspace{0.2cm}
	
	\end{center}

	\begin{abstract}
The main purpose of this paper is to study the existence of periodic solutions for a nonautonomous differential-difference system describing the dynamics of hematopoietic stem cell (HSC) population under some external periodic regulatory factors at the cellular cycle level. The starting model is a nonautonomous system of two age-structured partial differential equations describing the HSC population in quiescent ($G_0$) and proliferating ($G_1$, $S$, $G_2$ and $M$) phase. We are interested on the effects of a periodically time varying coefficients due for example to circadian rhythms or to the periodic use of certain drugs, on the dynamics of HSC population. The method of characteristics reduces the age-structured model to a nonautonomous differential-difference system. We prove under appropriate conditions on the parameters of the system, using topological degree techniques and fixed point methods, the existence of periodic solutions of our model.
\vspace{0.3cm} 
\vspace{0.3cm}\\
\textbf{AMS Math. Subj. Classification}:
34K13, 37C25, 37B55, 39A23
	\end{abstract}

\begin{section}{Introduction}
\subsection{Biological motivation}
The process that leads to the production and regulation of blood cells (red blood cells, white cells and platelets) to maintain homeostasis (metabolic equilibrium) is called hematopoiesis. The different blood cells have a short life span of one day to several weeks. The hematopoiesis process must provide daily renewal with very high output (approximately $10^{11}$-$10^{12}$ new blood cells are produced each day \cite{LeiMacJTB2011}). It consists of mechanisms triggering differentiation and maturation of hematopoietic stem cells (HSCs). Located in the bone marrow, HSCs are undifferentiated cells with unique capacities of differentiation (the ability to produce cells committed to one of blood cell types) and self-renewal (the ability to produce identical cells with the same properties) \cite{WilLauOseEtAllCP2008}. Cell biologists classify HSCs, \cite{BurTanCP1970}, as proliferating (cells in the cell cycle: $G_1$-$S$-$G_2$-$M$-phase) and quiescent (cells that are withdrawn from the cell cycle and cannot divide: $G_0$-phase). Quiescent cells are also called resting cells. The vast majority of HSCs are in quiescent phase \cite{BurTanCP1970,WilLauOseEtAllCP2008}. Provided they do not die, they eventually enter the proliferating phase. In the proliferating phase, if they do not die by apoptosis, the cells are committed to divide a certain time after their entrance in this phase. Then, they give birth to two daughter cells which, either enter directly into the quiescent phase (long-term proliferation) or return immediately to the proliferating phase (short-term proliferation) to divide again \cite{FicMurLinCleCSC2008,VegWinMelIL2010,WilLauOseEtAllCP2008}. 

The first mathematical model for the dynamics of HSCs was proposed by Mackey in 1978 \cite{MacB1978}. He proposed a system of delay differential equations for the two types of HSCs, proliferating and quiescent cells. Several improvements to this model has been made by many authors. In many of these works, it is assumed that after mitosis, all daughter cells go to the quiescent state. In a recent work by M. Adimy, A. Chekroun, and T.M. Touaoula \cite{AdiCheTouDCDSB2015}, a model was proposed that takes into account the fact that only a fraction of daughter cells enter the quiescent phase (long-term proliferation) and the other fraction of cells return immediately to the proliferating phase to divide again (short-term proliferation). This assumption leads to an important difference in the mathematical treatment of the model: it can no longer be posed as a system of delay differential equations. The system of equations has a different mathematical nature. An extra variable is introduced whose dynamics are ruled by a difference equation (no derivative involved). 

It is believed that several hematological diseases are due to some abnormalities in the feedback loops between different compartments of hematopoietic populations \cite{FolMacJMB2009}. These disorders are considered as major suspects in causing periodic hematological diseases, such as chronic myelogenous leukemia \cite{AdiCraRuaSIAMJAM2005,ColMacJTB2005,ForMacJH1999,PujBerMacJADS2005,PujMacCRB2004}, cyclical neutropenia \cite{ColMacJTB2005Bis,HauDalMacB1998,HauPerDalMacEH1999}, periodic auto-immune hemolytic anemia \cite{MacBMB1979,MilMacJRCP1989}, and cyclical thrombocytopenia \cite{ApoMacJTB2008,SanBelMahMacJTB2000}. In some of these diseases, oscillations occur in all mature blood cells with the same period; in others, the oscillations appear in only one or two cell types. The existence of oscillations in more than one cell line seems to be due to their appearance in HSC compartment. That is why the dynamics of HSC have attracted attention of modelers for more than thirty years now (see the review of C. Foley and M.C. Mackey \cite{FolMacJMB2009}). On another side, as for  most human cells, the circadian rhythm orchestrates the daily rhythms of HSCs \cite{ClaMicPerCRM2006}. It consists of a set of events that regulates DNA synthesis and mitotic activity \cite{BerHanGI2006,BjaJorSotAJP1999,PotBooCraEtAllCP2002,SmaLaeLotEtAllB1991}, and on a genetic level, tumor suppression \cite{FuPelLiuEtAllC2002}, and DNA damage control \cite{GerKomBalEtAllMC2006}. Molecular mechanisms underlying circadian control on apoptosis and cell cycle phases through proteins such as p53 and the cyclin-dependent kinase inhibitor p21 are currently being unveiled \cite{ClaMicPerCRM2006,FuPelLiuEtAllC2002,MatYamMitEtAllS2003}. The circadian fluctuations create periodic effects on the dynamics of cell population which promote certain times of cell division \cite{ClaMicPerCRM2006}. This phenomenon contributes to the emergence of cells with specific cell cycle durations which could play a role in promoting tumor development and at the same time, allowed the establishment of strategies for the treatment of cancer. The assumption of the periodicity of the parameters in the system incorporates the periodicity of the extracellular factors (extracellular proteins and various constituent components of the temporally oscillatory environment). For this reason, the assumption of periodicity is an approximation of the fluctuation of environmental factors. In fact, several different periodic models have been studied (see for instance, \cite{ClaMicPerCRM2006,DiaZhoAMC2016,DinLiuNieAMM2016,KapKhuIJDE2019,WenCMA2002,XuLiADE1998,YaoJNSA2015,ZhaYanWanAML2013,ZhoWanZhoAAA2013,ZhoYanJMAA2018}).

We will consider some of the key aspects of our model and briefly review the results obtained in \cite{AdiCheTouDCDSB2015}. In particular, we shall focus on the existence of equilibria and their stability properties. In this paper, a further generalization is considered, in order to take into account some external periodic regulatory factors at the cellular cycle level, by allowing some of the constants of the model, $\delta$, $K$ and $\gamma$, to be time $T$-periodic functions. This introduces further mathematical complexity since now the system of equations is nonautonomous. Some of the results of \cite{AdiCheTouDCDSB2015} extend in a straightforward manner. Others, like the equilibria under different regimes of parameters, change to other kind of structures in the nonautonomous setting. More precisely, we will show using topological techniques that our extended model exhibits periodic solutions under similar assumptions to those guarantying existence of a non-trivial equilibrium in \cite{AdiCheTouDCDSB2015}.

\subsection{Autonomous mathematical model of HSC dynamics}

Let us present the model introduced in \cite{AdiCheTouDCDSB2015}. Denote by $q(t, a)$ and $p(t, a)$ the population density of quiescent HSCs and proliferating HSCs respectively, at time $t\geq0$ and age $a\geq0$. The age represents the time spent by a cell in its current state. Quiescent cells can either be lost randomly at a rate $\delta \geq 0$, which takes into account the cellular differentiation, or enter into the proliferating phase at a rate $\beta \geq 0$. A cell can stay its entire life in the quiescent phase, therefore its age $a$ ranges from $0$ to $+\infty$. In the proliferating phase, cells stay a time $\tau \geq 0$, necessary to perform a series of processes, $G_1$, $S$, $G_2$ and $M$, leading to division at mitosis. Meanwhile they can be lost by apoptosis (programmed cell death) at a rate $\gamma \geq 0$. At the end of proliferating phase, that is, when cells have spent a time $a = \tau$, each cell divides in two daughter cells. A part $K \in (0,1)$ of daughter cells returns immediately to the proliferating phase to go over a new cell cycle while the other part ($1- K$) enters directly the resting phase. This dynamic is depicted in Figure \ref{Fig1}. 
\begin{figure}
	\vspace{0.5cm}
\begin{center}
\includegraphics[width=0.75\linewidth]{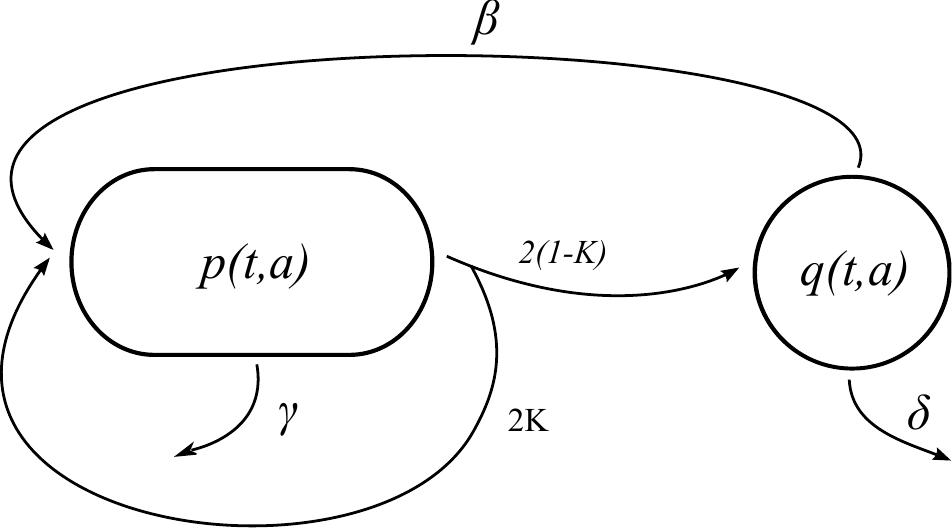}
\caption{Dynamic of HSCs (see, \cite{AdiCheTouDCDSB2015})}\label{Fig1}
\end{center}
\end{figure}
Consider $Q(t)=\int_0^{+\infty} q(t,a)da$ and $P(t)=\int_0^{\tau} p(t,a)da$  the total populations at a given time $t\geq 0$, and $u(t):=p(t,0)$ the number of cells entering the proliferating state at a given time $t\geq 0$. The rate $\beta$ depends on $Q(t)$ in a nonlinear way, by a Hill function (see \cite{MacB1978}), 
\[\beta(Q):=\dfrac{\beta_0}{1+Q^r}, \qquad \beta_0>0, \; r>1.\]
The partial differential equations for this age-structured model read, for  $t
\geq 0$,
\begin{equation}\label{1}
\left\{
\begin{array}{rcll}
	q_t + q_a &=& -(\delta + \beta (Q(t)))q, & a \in [0,+\infty),  \vspace{0.1cm}\\
	p_t + p_a &=& -\gamma p,    & a \in [0,\tau],\vspace{0.1cm}\\
	q(t,0)&=& 2(1-K)p(t,\tau), & \vspace{0.1cm}\\
	p(t,0)&=&\beta(Q(t))Q(t)+2Kp(t,\tau), & \\
\end{array}
\right.
\end{equation}
with initial conditions
\begin{equation}\label{IC}
\begin{cases}
     q(0,a)=q_0(a), \qquad a\in [0,+\infty), \\
     p(0,a)=p_0(a), \qquad a\in [0,\tau],
\end{cases}
\end{equation}
and the following natural condition
\[\lim_{a\to +\infty}q(t,a)=0.\]
Using the method of characteristics (see \cite{AdiCheTouDCDSB2015}), we get for $t>\tau$
\[p(t,\tau)=e^{-\gamma \tau}p(t-\tau,0).\]
Integrating the system \eqref{1} with respect to the age $a$ and putting \[u(t)=\varphi(t):=e^{-\gamma t}p_0(-t), \quad t\in [-\tau,0],\]  
yields the following system, for $t>0$, 
\begin{empheq}[left=\empheqlbrace]{align}
	Q'(t) &= -(\delta + \beta(Q(t)))Q(t) +2(1 -K)e^{-\gamma \tau}u(t-\tau), \label{2}\\
	P'(t) &= -\gamma P(t) + \beta (Q(t)) Q(t) - (1 - 2K)e^{-\gamma \tau}u(t - \tau ),\\
	u(t) &= \beta(Q(t))Q(t) + 2Ke^{-\gamma \tau}u(t-\tau),\label{3}
\end{empheq}
with initial conditions
\[Q(0)=Q_0:=\int_0^{+\infty} q_0(a)da, \quad P(0)=P_0:=\int_0^{\tau} p_0(a)da\]
and
\[u(t)=\varphi(t), \quad t\in [-\tau,0].\]
Remark that $P$ can be recovered from $u$, namely, 
\[P(t)=\int_0^{\tau} e^{-\gamma a}u(t-a)da, \quad t\geq 0.\] 
On the other hand, the two equations satisfied by $Q$ and $u$ are independent of $P$. So, it suffices to analyze the reduced system for $Q$ and $u$ only. It should be noted that the equation for $u$ is not differential. This fact poses a difficulty in using some of the standard topological methods, because the right inverse of the linear operator associated to the equation of $u$ is not compact. 
The reduced system reads
\begin{empheq}
[left=\empheqlbrace]{align}
Q'(t) &=-(\delta+\beta(Q(t))Q(t)+2(1-K)e^{-\gamma\tau}u(t-\tau),\label{A1}\\
u(t) &=\beta(Q(t))Q(t)+2Ke^{-\gamma\tau}u(t-\tau).\label{A2}
\end{empheq}
The following set of hypotheses can be regarded as ``natural'' in the context of the model.
\begin{description}
\item[\Hy{(H0)}\label{H0}] $\delta$, $K$ and $\gamma$ are positive parameters, $0<K<1$ and $\beta(Q):=\dfrac{\beta_0}{1+Q^r}$, with $\beta_0>0$ and $r>1$.
\end{description}

In order to express our conditions for existence of solutions in an accurate way, let us define the following quantities: 
 \begin{equation*}
\begin{array}{cccc}
     h_1:=2(1-K)e^{-\gamma\tau}, & h_2:=2Ke^{-\gamma\tau}, & \alpha:=\dfrac{h_1}{1-h_2}-1.  
\end{array}  
 \end{equation*}

Also, for $Q>0$ we define the function $\jj (Q):=\beta(Q)Q$, which attains a global maximum $B:=\max_{Q>0} \jj(Q)$. 

\begin{center}
\begin{tikzpicture}[scale=0.6]
\begin{axis}[
    xmin=0,
    ymin=0,
    width=5in,
    legend style={legend pos=north east,font=\large,row sep=1cm,draw=none},
    axis line style={latex-latex},
    xtick={\empty},
    extra y ticks={2,1.05820788459},
    extra y tick labels={$\beta_0$,$B$},
    ytick={\empty},]
    
\addplot[
    samples=100, 
    smooth,
    color=red, 
    domain=0:10,
    ]
{2/(1+x^3)};
\addlegendentry{$\beta(Q)=\dfrac{\beta_0}{1+Q^r}$}.

\addplot[samples=100, 
    smooth, color=blue, domain=0:10]{2*x/(1+x^3)};
\addlegendentry{$\jj(Q)=\dfrac{\beta_0 Q}{1+Q^r}$}.

\end{axis}

\end{tikzpicture}%
\end{center}

The following results were proven in \cite{AdiCheTouDCDSB2015}.

\begin{theorem}\label{teorema1}
System \eqref{A1}-\eqref{A2} has a nontrivial equilibrium (\QQ,\uu) iff
\begin{description}
\item[\Hy{(H1)}\label{H1}] $h_2< 1$ (whence $\alpha<\infty$), 
\item[\Hy{(H2)}\label{H2}] $\alpha> 0$, 
\item[\Hy{(H3)}\label{H3}] $\delta<\alpha \beta_0$. 
\end{description}
In that case, the nontrivial equilibrium is given by
\[(\QQ,\uu)=\left(\beta ^{-1}\left( \frac{\delta}{\alpha}\right),\frac{\delta}{2e^{-\gamma\tau}}\beta ^{-1}\left(\frac{\delta}{\alpha}\right)\right).\]
\end{theorem}
We remark that as the parameters $\delta$ and $\beta_0$ are positive, the assumption \ref{H3} implies \ref{H2}. Furthermore, the assumptions \ref{H1}-\ref{H2} are equivalent to 
\[\max\left\lbrace \dfrac{1}{\gamma} \ln(2K),0\right\rbrace<\tau<\dfrac{1}{\gamma}\ln(2)\]
and the condition \ref{H3} is equivalent to  
\[\tau<\dfrac{1}{\gamma}\ln\left( \dfrac{2(\beta_0+\delta K)}{\beta_0+\delta}\right).\]
\begin{theorem}\label{teorema2}
Assume that \ref{H1}-\ref{H2} and the following condition
 \begin{description}
	\item[\Hy{(H3')}\label{H3'}]  $\delta>\alpha \beta_0$, 
\end{description}
are satisfied. Then, the trivial equilibrium is globally asymptotically stable.
\end{theorem}
We remark that the assumption \ref{H3'} is equivalent to 
\[\tau>\dfrac{1}{\gamma}\ln\left( \dfrac{2(\beta_0+\delta K)}{\beta_0+\delta}\right).\]

\subsection{Nonautonomous model of HSC dynamics}
In this work, we shall consider a nonautonomous case, 
with $K$, $\gamma$ and $\delta$ 
continuous
$T$-periodic functions, that is for  $t \in [0,+ \infty)$, 
\begin{equation}\label{8}
\left\{
\begin{array}{rcll}
	q_t + q_a &=& -(\delta(t) + \beta (Q(t)))q, & a\in [0,+\infty),  \vspace{0.1cm}\\
	p_t + p_a &=& -\gamma(t) p,    & a \in [0,\tau],\vspace{0.1cm}\\
	q(t,0)&=& 2(1-K(t))p(t,\tau), &\vspace{0.1cm}\\
	p(t,0)&=&\beta(Q(t))Q(t)+2K(t)p(t,\tau). &\\
\end{array}
\right.
\end{equation}
Using again the method of characteristics, we obtain
\[p(t,\tau)=\exp\left(-\int_{t-\tau}^{t} \gamma(s)ds \right) p(t-\tau,0), \quad \text{for} \;\; t\geq\tau.\]
For convenience, set
\[\rho(t)=\int_{t-\tau}^{t} \gamma(s)ds, \quad t\geq\tau,\]
 which is also a  $T$-periodic function. As for the system \eqref{1}, the age-structured partial differential model \eqref{8} can be reduced to 
\begin{empheq}
[left=\empheqlbrace]{align}
Q'(t) &=-(\delta(t)+\beta(Q(t))Q(t)+2(1-K(t))e^{-\rho(t)}u(t-\tau),\label{Eq1}\\
u(t) &=\beta(Q(t))Q(t)+2K(t)e^{-\gamma\tau}u(t-\tau).\label{Eq2}
\end{empheq}
For convenience, we define as before 
\[h_1(t):=2(1-K(t))e^{-\rho(t)}\] 
and 
\[h_2(t):=2K(t)e^{-\rho(t)},\]
which turn out to be $T$-periodic functions. Also, we define the quantity
\[\alpha:=\dfrac{\min(h_1)}{1-\min(h_2)}-1.\]
Our basic hypothesis now reads as follows.
\begin{description}
\item[\Hy{(H0)}\label{H0}] $\delta$, $\gamma$ and $K$ are positive $T$-periodic functions, $\max(K)<1$ and $\beta(Q)=\dfrac{\beta_0}{1+Q^r}$ with $\beta_0>0$ and $r>1$.
\end{description}

\subsection{Main results}\label{subsection 1.4}

Three results will be presented in this work. In the first place, we shall prove the existence of $T$-periodic solutions of \eqref{Eq1}-\eqref{Eq2} under appropriate conditions on the functions $\delta$, $\gamma$ and $K$. 
\begin{theorem}\label{teoremappal}
Assume that \ref{H0} holds and 
\begin{description}
\item[\Hy{(H1)}\label{H1}] $h_2(t)<1$, for all $t \in \R$,
\item[\Hy{(H2)}\label{H2}] $\alpha> 0$,
\item[\Hy{(H3)}\label{H3}] $\delta(t)<\alpha \beta_0$, for all $t \in \R$.
\end{description}
Then, \eqref{Eq1}-\eqref{Eq2} has at least one positive $T$-periodic solution.
\end{theorem}

For a  proof, we shall rewrite the system \eqref{Eq1}-\eqref{Eq2} as a single equation for $Q$. Thus, 
solutions can be obtained as  the zeros of a conveniently defined operator over the Banach space of continuous $T$-periodic functions. 
We guarantee the existence of at least one nontrivial zero by means of the Leray-Schauder degree theory. 
We remark that, in contrast with other methods (\textit{i.e.} using the contraction mapping theorem), the Leray-Schauder continuation method gives no information about 
the uniqueness of such periodic solution or its amplitude. 

In the second place, we shall study small perturbations of 
the autonomous system. In more precise terms, assume the  conditions of Theorem $\ref{teorema1}$ are satisfied and 
consider small $T$-periodic perturbations of the parameters. 
It would be natural to expect that the nontrivial equilibrium is then perturbed into a $T$-periodic solution of small amplitude oscillating close to such equilibrium. In order to formalize such intuition, consider the continuous $T$-periodic vector function $\Lambda=(\delta, K,\gamma)\in\C^3$, with $\C^3$ the Banach space of continuous $T$-periodic functions. Thus, \eqref{Eq1}-\eqref{Eq2} can be thought as a parametric system of equations with parameters defined in $\C^3$. For convenience, the subset of constant functions in $\C^3$ shall be identified with $\mathbb R^3$.
This setting includes both the autonomous and nonautonomous systems and allows to introduce our second result as follows.
\begin{theorem}
Assume that a constant parameter $\underline{\Lambda}\in \mathbb R^3$ and the delay $\tau$ satisfy appropriate conditions (to be specified),
then there exist open subsets $U\subset \C^3$  with  $\underline{\Lambda}\in U$ and $V\subset \C$, and a continuous map $I: U \to V$ such that $I(\Lambda)$ is a $T$-periodic solution of  the system \eqref{Eq1}-\eqref{Eq2} with continuous $T$-periodic vector function $\Lambda$. Moreover, $I(\Lambda)$ is unique in $V$.
\end{theorem}
The preceding theorem gives also  a way to obtain periodic solutions; in some sense, it provides a better characterization of such solutions. We remark, however, 
that the sufficient conditions for existence 
are explicit in the first result and not in the second one. 

Finally, our last result extends Theorem $\ref{teorema2}$ to the nonautonomous case.
\begin{theorem}
Assume that \ref{H1}-\ref{H2} and the following condition \begin{description}
    \item[\Hy{(H3')}\label{H3'}]  $\delta(t)>\alpha(t) \beta_0$, for all $t \in \R$, 
\end{description}
are satisfied. Then, the trivial equilibrium is locally asymptotically stable.
\end{theorem}
It is worth mentioning that the latter theorem is local and, consequently, it does not imply that nontrivial periodic solutions cannot exist. However, if such solutions exist, then they are necessarily ``large''. An explicit subset of  the basin of attraction of the trivial equilibrium shall be characterized in the proof.

\section{First result}
\subsection{Sketch of the proof}

For the reader's convenience, let us firstly sketch the idea of the proof.

Due to the above mentioned lack of compactness, we shall reduce the problem to a scalar equation in the following way. Set $C_T$ as the Banach space of continuous $T$-periodic functions and $\mathcal C\subset C_T$ the cone of nonnegative functions.  
Given $Q\in \mathcal C$, we shall prove the existence of a unique  solution $\mathbf{u} (Q)$ of \eqref{Eq2} and, furthermore, that the mapping $\mathbf{u}: \mathcal C \mapsto C_T$ is continuous. Thus, finding a $T$-periodic solution of the system is equivalent to solve the problem
\begin{equation}\label{B}
Q'=N(Q):=\mathbf N(Q,\mathbf{u}(Q)),
\end{equation}
in $\mathcal C$, where $\mathbf N(Q,u)$ is the Nemytskii  operator associated to  the right-hand side of the equation \eqref{Eq1}. Once a $T$-periodic solution $Q$  of \eqref{B} is found, the pair $(Q,\mathbf{u}(Q))$ is a $T$-periodic solution for the system \eqref{Eq1}-\eqref{Eq2}.

For the scalar equation \eqref{B}, we shall apply 
the continuation method over 
a bounded open set of the form $\Omega_{\epsilon,R}=\lbrace Q\in C_T: \epsilon< Q(t)< R\rbrace$, with 
$R>\epsilon  >0$  chosen in such a way that  $\Omega_{\epsilon,R}$ satisfies the hypotheses of Mawhin's continuation Theorem (see \cite{BroS1993}). For convenience, the ideas behind this result (degree theory, Lyapunov-Schmidt reduction) shall be briefly discussed in the next section.

\subsection{Mawhin's continuation Theorem}

For the sake of completeness, let us recall some facts about the degree theory that shall be employed in our proof. The Leray-Schauder degree is an infinite dimensional extension of the Brouwer degree $\db$ of a continuous function. We shall define $\dls$ for operators on a Banach space $B$ that are compact perturbations of the identity. In more precise terms, let $\Omega\subset B$ be open and bounded and $\F: \overline\Omega\subset B \to B$ such that $\F=Id+\mathcal{C}$ with $\mathcal{C}$ compact. We will just give a brief summary of the properties that shall be used in this work (for more details on the degree theory see for example \cite{BroS1993}).
\begin{proposition}
Let $\mathcal{C}:\overline\Omega \to B$ be a compact mapping. Then, there exists a sequence of mappings $\mathcal{C}_{n}$ 
of finite rank that approximates $\mathcal{C}$ uniformly over $\overline\Omega$.
\end{proposition} 
This allows the following definition.
\begin{definition}
Let $\Omega \subset B$ be bounded and open subset such that 
$\F=Id+\mathcal{C}$ does not vanish on $\partial\Omega$ and define
\[\dls(Id+\mathcal{C},\Omega,0):=\db(\restr{Id+\mathcal{C}_{\V}}{\Omega\cap\V},\Omega\cap \V,0),\]
where $\mathcal{C}_{\V}$ is sufficiently close finite rank approximation of $\mathcal{C}$ with rank contained in $\V$. 
\end{definition}
It can be proven that the definition does not depend on the choice of $\mathcal{C}_\V$ (see Theorem 9.4, page 60 of \cite{BroS1993}). The  following properties shall be fundamental for our purposes.
\begin{proposition}
If $\dls(\F,\Omega,0)\neq0$ then $\F$ has a zero in $\Omega$.
\end{proposition}

\begin{definition}
We say that the family of operators $\lbrace \F_\lambda \rbrace_{0\leq\lambda\leq1}$ is an admissible homotopy over a set $\Omega$ if and only if 
\begin{itemize}
\item $\mathcal{F}_{\lambda}=Id-\mathcal{C}_{\lambda}$, with 
$\mathcal{C}_{\lambda}=\mathcal C(\cdot,\lambda)$
and $\mathcal C:\overline{\Omega}\times [0,1]\to B$
continuous and compact.
\item $\mathcal{F}_{\lambda}(Q)\neq 0$, for all $Q\in\partial\Omega$ 
and for all  $\lambda\in[0,1]$.
\end{itemize}
\end{definition}

\begin{proposition}
If $\lbrace \F_\lambda \rbrace_{0\leq\lambda\leq 1}$ is an admissible homotopy over a set $\Omega$, then $\dls(\F_\lambda,\Omega,0)$ is constant with respect to $\lambda$.
\end{proposition}

\begin{proposition}\label{grado en dimension 1}
The Brouwer degree is easily computable in the one-dimensional case; specifically, when $\Omega =(a,b)$ it is seen that
\begin{equation*}
    \db(f,\Omega,0)= 
    \begin{cases}
       -1, \quad & \ \text{if} \;\; f(a)>0 \ \text{and} \ f(b)<0,\\
       0, \quad & \ \text{if} \;\; f(a)f(b)>0, \\
       1, \quad  & \ \text{if} \;\; f(a)<0 \ \text{and}  \ f(b)>0.\\
     \end{cases}
\end{equation*}
\end{proposition}
Let $C^1_T=C^1\cap C_T$, and let
\[\overline{Q}=\frac{1}{T}\int_0^T Q(t)dt\]
denotes the average of a function $Q$. The set of constant functions shall be identified with $\R$. The following result by J. Mawhin (see \cite{BroS1993}), adapted for our purposes, sums up the technique that shall be  used to prove the existence theorem.

\begin{lemma}\label{Mawhin}
Assume $N:C_T \to C_T$ is a continuous nonlinear operator and $\Omega \subset C_T$ is an open bounded set. Consider the equation
\begin{equation}\label{eq del lemma}
Q'=N(Q).
\end{equation}
For a constant function $Q\equiv q$, define $f(q)=\overline{N(Q)}$ and assume that the following conditions hold:
\begin{enumerate}
\item\label{condicion 1 Mawhin} $Q'=\lambda N(Q)$ has no solutions on $\partial \Omega$, for $\lambda \in (0,1)$;
\item\label{condicion 2 Mawhin} $f(q)\neq 0$, for $q\in \partial\Omega \cap \R$;
\item $\db (f, \partial\Omega \cap \R)\neq 0$.
\end{enumerate}
Then, there exists a $T$-periodic solution of the equation \eqref{eq del lemma} with range in $\overline\Omega$. 
\end{lemma}
\begin{proof}
For $f\in C_T$ such that $\overline f=0$, define \[\K(f)(t):=\int_{0}^{t}f(s)ds-\int_{0}^{T}f(s)ds.\]
It is immediate that $\K(f)\in C_T^1$ and $(\K(f))'=f$. So, $\K$ is a right inverse of the differentiation operator. A straightforward application of Arzel\'a-Ascoli Theorem shows that  $\K$ is compact. The operators 
\[\F_\lambda(Q):=Q-\overline{Q}+\overline{N(Q)}-\lambda \K\left( N(Q)-\overline{N(Q)}\right) \]
are well defined compact perturbations of the identity and, for $0<\lambda\leq 1$,
\begin{equation*}
	\mathcal{F}_{\lambda}(Q)=0 \qquad \text{if and only if} \qquad Q'=\lambda N(Q).
\end{equation*}
Thus, a zero of $\F_1$ is a $T$-periodic solution of the equation \eqref{eq del lemma}. 
The first assumption of the lemma means exactly that $\F_\lambda$ is an admissible homotopy in $\Omega$. The second and third conditions 
correspond to the well definition of $\dls(\F_0, \Omega,0)$ and  the fact that $\dls(\F_0, \Omega,0)\neq 0$. The invariance of $\dls$ under homotopies completes the proof.
\end{proof}

\subsection{The mapping $\mathbf{u}$}

Let us recall that our method consists in reducing the system \eqref{Eq1}-\eqref{Eq2} to a scalar equation for $Q$, for which Lemma \ref{Mawhin} can be applied. In order to do so, it needs to be shown that, for given $Q\in C_T$, there exists a unique $\mathbf{u} (Q)$ solution of \eqref{Eq2}. This shall define a mapping $\mathbf{u}: C_T \mapsto C_T$. The following lemma proves that such mapping exists and is continuous. Further, it also gives estimates on the image of some set of the form
\[\Omega_{\epsilon,R}=\lbrace Q\in C_T: \epsilon< Q(t)< R\rbrace,\]
that will be employed in  the continuation Lemma. 

\begin{lemma}\label{lemma1}
Assume that the hypothesis \ref{H1} of Theorem \ref{teoremappal} holds. Then, given $Q\in C_T$, there exists a unique  solution $\mathbf{u} (Q)$
 of \eqref{Eq2}. The mapping $\mathbf{u}: C_T \mapsto C_T$ is continuous. Moreover, if $0<\epsilon<R$ are such that $j(\epsilon)<j(R)$, then 
 \[
 \mathbf{u}(\Omega_{\epsilon,R})\subset\mathcal{U}_{\epsilon}:=\left \{ u\in C_T: \dfrac{\beta(\epsilon)\epsilon}{1-\min(h_2)}\leq u \leq \dfrac{B}{1-\max(h_2)}\right \}.
 \] 
\end{lemma}
\begin{proof}
Define $S(u)(t):=u(t)-h_2(t)u(t-\tau)$. Then, the equation \eqref{Eq2} can be written as
\[S(u)=\jj\circ Q.\]
The norm of $(Id-S)(u)(t)=h_2(t)u(t-\tau)$ in the space $\mathcal{L}(C_T)$ of linear operators on $C_T$ is computed from the inequality
\[| (Id-S)(u)(t)|=| h_2(t)u(t-\tau)| \leq \max(h_2) \ |u(t-\tau)| \leq \max(h_2)\norm{u}_{C_{T}},\]
which implies
\[\| Id-S\|\le  \max(h_2) <1.\]
As a consequence, $S$ is invertible with continuous inverse. Hence, the mapping $\mathbf{u}(Q)=S^{-1}(\jj\circ Q)$ is well defined and continuous.

In order to find estimates for $\mathbf{u}(Q)$ in terms of the estimates on $Q$, we will follow a roundabout way. Given a fixed $Q\in C_T$, let us define $S_Q(u)(t)=\jj(Q(t))+h_2(t)u(t-\tau)$. Solving the equation \eqref{Eq2} for $Q$, is equivalent to find a fixed point of $S_Q$. Next observe that, given any $Q\in C_T$, the mapping $S_Q$ is a contraction. So, by the Banach Fixed Point Theorem it has a unique fixed point, which is necessarily equal to $\mathbf{u}(Q)$. This gives us another way to characterize $\mathbf{u}(Q)$.

Now, let $Q\in \Omega_{\epsilon,R}$. If we could find an invariant set $\mathcal{U}$ for $S_Q$ then, by Banach's Theorem, the (unique) fixed point $\mathbf{u}(Q)$ will belong to $\mathcal U$. 
 With this idea in mind, consider sets of the form $\mathcal{U}_{a,b}:=\lbrace u \in C_T: a\leq u \leq b \rbrace$. It follows from the hypothesis that the minimum value of $\jj$ in $[\epsilon, R]$ is attained at $\epsilon$. Suppose that $Q\in\Omega_{\epsilon,R}$, then given $u\in \mathcal{U}_{a,b}$, we have
\[\jj(\epsilon)+a \ \min(h_2)  \leq  \jj(Q(t))+h_2(t)u(t-\tau)\leq B+ b \ \max(h_2).\]
Hence, taking $a=\dfrac{\jj(\epsilon)}{1-\min(h_2)}$ and $b= \dfrac{B}{1-\max(h_2)}$ it is deduced that $S_Q(\mathcal{U}_{a,b})\subseteq\mathcal{U}_{a,b}$. 
So, for $Q\in \Omega_{\epsilon,R}$, $\mathbf{u}(Q) \in \mathcal{U}_{\epsilon}$. This means that $\mathbf{u}(\Omega_{\epsilon,R})\subset\mathcal{U}_{\epsilon}$.
\end{proof}

\subsection{Proof of Theorem \ref{teoremappal}}

We are now in condition of  proving our existence theorem.
To this end, we shall show that \ref{H1}, \ref{H2} and \ref{H3} (of Theorem \ref{teoremappal}) allow to find $\epsilon$ and $R$ such that the assumptions of Lemma \ref{Mawhin} are satisfied for $\Omega_{\epsilon,R}$.

Since $\beta(R)\to 0$ and $\frac{1}{R} \frac{B\max (h_1)}{1-\max (h_2)} \to 0$ as $R\to +\infty$, we may choose $R$ large 
enough such that $\min (\delta)>-\beta(R)+\frac{1}{R} \frac{B\max (h_1)}{1-\max (h_2)}$. Once $R$ is chosen, 
using \ref{H3} and the fact that $\jj(\epsilon) \to 0$ as $\epsilon \to 0$, we may choose $\epsilon$ small enough such that $1+\epsilon^{r}<\dfrac{\beta_0\alpha}{\max (\delta)}$ 
 and also $\jj (\epsilon)<\jj(R)$. Summarizing, our choice of  $\epsilon$ and $R$ yields:
 \begin{description}
\item[\Hy{(C0)}\label{C0}] $0<\epsilon<R$ and $\jj(\epsilon)<\jj(R)$, 
\item[\Hy{(C1)}\label{C1}]  $1+\epsilon^{r}<\dfrac{\beta_0\alpha}{\max (\delta)}$,
\item[\Hy{(C2)}\label{C2}]  $\min (\delta)>-\beta(R)+\dfrac{1}{R} \dfrac{B\max (h_1)}{1-\max (h_2)}$.
\end{description}
Let us check now that for such $\epsilon$ and $R$, the first condition in Lemma \ref{Mawhin} is satisfied.

Let $\lambda\in (0,1)$ and suppose there exists $Q\in \partial\Omega_{\epsilon,R}$ such that $Q'=\lambda N(Q)$. The fact that $Q\in \partial \Omega_{\epsilon,R}$ implies there exists $t_0 \in [0,T]$ such that $Q(t_0)=\epsilon$, or such that $Q(t_0)=R$. If $Q(t_0)=\epsilon$, then, $Q$ reaches its minimum value at $t_0$ and hence $0=Q'(t_0)=\lambda N(Q(t_0))$. That is,
\begin{eqnarray*}
0&=&-(\delta(t_0)+\beta(\epsilon))\epsilon+h_1(t_0)u_{Q}(t_0-\tau),\\
\delta(t_0)\epsilon & =&-\beta(\epsilon)\epsilon+h_1(t_0)\mathbf{u}(Q)(t_0-\tau).\\
\end{eqnarray*}
Using \ref{C0} and the fact that $Q\in \Omega_{\epsilon,R}$, we may apply Lemma \ref{lemma1} in order to get
\[\delta(t_0) \epsilon  \geq -\beta(\epsilon)\epsilon + \
\min(h_1) \dfrac{\beta(\epsilon)\epsilon}{1-\min(h_2)}.\]
Thus, 
\begin{equation}\label{D}
\delta(t_0) \geq \beta (\epsilon)\left \{   \frac{\min(h_1)}{1-\min(h_2)} - 1\right \}=\dfrac{\beta_0\alpha}{1+\epsilon^r}.
\end{equation}
This contradicts \ref{C1}.

Now suppose there exists $t_0$ such that $Q(t_0)=R$. Then,
by \ref{C0} and Lemma \ref{lemma1}, we obtain 
\begin{eqnarray*}
\delta(t_0)R & =&-\beta(R)R+h_1(t_0)\mathbf{u}(Q)(t_0-\tau),\\
 & \leq &-\beta(R)R +\dfrac{B\max (h_1)}{1-\max (h_2)}.
\end{eqnarray*}
This contradicts \ref{C2} and the first condition of Lemma \ref{Mawhin} is thus proven.

Next, we shall verify  the second condition. In the first place, notice that $\Omega_{\epsilon,R} \cap \R=[\epsilon,R]$. Now, suppose $f(q)=\overline{N(Q)}=0$, for some $Q\equiv q \in \partial\Omega_{\epsilon,R} \cap \R=\lbrace \epsilon,R \rbrace$. Then, $Q\equiv \epsilon$ or $Q\equiv R$. In the first case, the fact that $\overline{N(\epsilon)}=0$ implies
\begin{eqnarray*}
0 & = \overline{-(\delta(t)+\beta(\epsilon))\epsilon + h_1(t)\mathbf{u}(Q)(t-\tau)},\\
\overline{\delta}& =-\beta(\epsilon)+\dfrac{1}{\epsilon}\overline{h_1(t)\mathbf{u}(Q)(t-\tau)}\geq\dfrac{\beta_0\alpha}{1+\epsilon^{r}}.
\end{eqnarray*}
But, $\max(\delta) \geq \overline{\delta}\geq\dfrac{C\alpha}{1+\epsilon^{r}}$ contradicts \ref{H1}. On the other hand, if $Q\equiv R$ then $\overline{N(Q)}=0$ implies
\begin{eqnarray*}
\overline{\delta}& =-\beta(R)+\dfrac{1}{R}\overline{h_1(t)\mathbf{u}(Q)(t-\tau)}\leq
-\beta(R)+\dfrac{1}{R}\dfrac{B\max (h_1)}{1-\max (h_2)}.
\end{eqnarray*} 
This contradicts \ref{H2}. 

It remains to show that the last condition in Lemma \ref{Mawhin} is satisfied. By Proposition \ref{grado en dimension 1}, the inequalities
\[f(\epsilon)>-(\max (\delta))\epsilon+\beta(\epsilon)\epsilon\alpha=\epsilon\left( \frac{\beta_0\alpha}{1+\epsilon^r}-\max (\delta)\right) >0\] 
and
\[f(R)<-\overline{\delta}R-\beta(R)R+R \dfrac{B\max (h_1)}{1-\max (h_2)}<0,\]
imply that $d_{B}(f,[\epsilon,R],0)=-1$.

Finally, using Lemma \ref{Mawhin}, we conclude the existence of $T$-periodic solution to equation \eqref{B}, which, in turn, gives us a $T$-periodic solution of system \eqref{Eq1}-\eqref{Eq2}.

\end{section}

\section{Second result}

\subsection{Preliminaries}
Consider the operator $\mathcal{F}:C_{T}\times \C^3 \to C_{T}$ given by
\begin{equation}
    \mathcal{F}(Q,\Lambda):=Q-\overline{Q}+\overline{N(Q,\Lambda)}+\K(N(Q,\Lambda)-\overline{N(Q,\Lambda)}).
\end{equation}
In other words, for each 
fixed $\Lambda\in\C^3$, the mapping $\mathcal{F}(\cdot,\Lambda)$ is  the operator defined in the proof of Lemma \ref{Mawhin}. 
We already know that for any constant $\underline{\Lambda}=(\underline{\delta},\underline{K},\underline{\gamma})$ 
satisfying the assumptions of Theorem \ref{teorema1}, there exists a (unique) stationary solution $\underline{Q}$. 
That is, under the previous identification of $\mathbb R$ with the set of constant functions, 
we have a pair $(\underline{Q},\underline{\Lambda})$ such that $\F(\underline{Q},\underline{\Lambda})=0$. 
We shall obtain a (locally unique) branch of solutions $Q(\lambda)$ when $\lambda$ is close to $\underline \lambda$ with the help of the Implicit Function Theorem, namely: 
\begin{theorem}
Let $X$, $Y$ and $Z$ be Banach spaces and let $U$ be an open subset of $X \times Y$. Let $\F$ be a continuously differentiable map from
$U$ to $Z$. If $(\underline{x}, \underline{y})\in U$  is a point such that $\F(\underline{x},\underline{y})=0$ and $D_{x}F(\underline{x}, \underline{y})$ is a bounded, invertible, linear
map from $X$ to $Z$, then there exist open neighborhoods $G$ and $H$ of $\underline y$ and $\underline x$, respectively, and a unique $C^1$ function $\varphi : G \to H$ such that $\varphi (\underline y)=\underline x$ and
$\F(\varphi(y), y) = 0$, for all $y \in G$.
\end{theorem} 
In more precise terms, if the Fr\'echet derivative of $\mathcal{F}$ with respect to $Q$ at the point $(\underline{Q},\underline{\Lambda})$ is an isomorphism, then, for all 
$\Lambda\in \C^3$ in a neighbourhood of $\underline{\Lambda}$ 
there exists a (locally unique) associated $T$-periodic function $Q$ and the mapping $\Lambda \mapsto Q$ is continuous.
This shows there is a continuity between the equilibrium provided by Theorem \ref{teorema1} and the periodic 
solutions $(Q,\mathbf{u}(Q))$ associated to small periodic perturbations of $\underline{\Lambda}$. In particular, these periodic solutions shrink to a point in the $(Q,u)$ plane, as the amplitude of the oscillations of $\Lambda$ goes to zero.

With this in mind, let us firstly recall that for any continuous linear operator $\mathcal{T}:C_{T}\times \C^3 \to C_{T}$ one has 
\[(D_{Q}\mathcal{T})(Q,\Lambda)\psi =\mathcal{T}\psi, \qquad \text{for all} \;\psi. \] 
Moreover, for an arbitrary operator $H$ we may write
 $\overline{H}=P\circ H$. So, by the chain rule we have
\[D_Q(\overline{H})=D_Q(P\circ H)=P\circ D_Q H=\overline{D_Q H}.\]
Let us compute $(D_{Q} \mathcal{F})(Q,\Lambda)$:
\begin{equation}
(D_{Q}\mathcal{F})(Q,\Lambda)\psi=\psi - \overline{\psi} + \overline{(D_{Q}N)(Q,\Lambda)\psi}+K\big((D_{Q}N)(Q,\Lambda)\psi-\overline{(D_{Q}N)(Q,\Lambda)\psi}\big).
\end{equation}

\begin{proposition}
If $\mathcal C:X\to Y$ is a compact (nonlinear) operator 
differentiable at $x_0$, then $D_x \mathcal C (x_0)$ is a compact linear operator.
\end{proposition}
\begin{proof}
See Theorem 14.1, page 96 of \cite{BroS1993}.
\end{proof}
From the previous computation and the last proposition 
we conclude that $(D_Q\F)(Q,\Lambda)$ is a compact perturbation of the identity (namely, a Fredholm operator of the type $I+\mathcal C$). Thus, in order to prove that it is an isomorphism, we only have to check its injectivity. To this end, observe that having an element $\psi$ in the kernel, means 
\begin{equation}\label{nucleo del diferencial}
\psi'=(D_{Q}N)(\underline{Q},\underline{\Lambda})\psi.    
\end{equation}
Next, recall that 
\begin{equation}
N(Q,\Lambda)=-(\delta+\beta(Q))Q+h_2(\Lambda)R_\tau \mathbf{u}(Q,\Lambda),
\end{equation}
where $R_\tau(\psi)(t)=\psi(t-\tau)$. Thus, 
\begin{equation}
(D_Q N) (Q,\Lambda)\psi=-(\delta + \ \jj '(Q))\psi + h_2(\Lambda)R_\tau (D_{Q}\mathbf{u})(Q,\Lambda)\psi.
\end{equation}

In order to compute the differential of $\mathbf{u}$, let us 
firstly clarify its definition. 
As shown before, given a fixed function $\Lambda$ satisfying \ref{H1}, it is possible to define an invertible operator $S$. This definition shall be now extended as follows. 
Let $\Lambda\subset \C^3$ the subset of $\Lambda$ satisfying \ref{H1}, then
\begin{equation} 
S:C_T \times \C^3 \to C_T, \quad S(u,\Lambda)=u+h_2(\Lambda)R_\tau u.
\end{equation}
For each fixed $\Lambda$, the operator $S_{\Lambda}(u):=S(u, \Lambda)$ is invertible and $\mathbf{u}(Q,\Lambda)=S_{\Lambda}^{-1}(\jj (Q))$ is continuous in $(Q,\Lambda)$ and differentiable in $Q$, with
\begin{equation}
(D_Q \mathbf{u})(Q,\Lambda)\psi =S_{\Lambda}^{-1}(\jj '(Q) \psi)=\jj '(Q)S_{\Lambda}^{-1}(\psi).
\end{equation}
So, the equation \eqref{nucleo del diferencial} reads 
\begin{align*}
\psi'&=-\underline{\delta} \psi-\jj '(\underline{Q})\psi + h_1(\underline{\Lambda})\jj '(\underline{Q})R_\tau S_{\underline{\Lambda}}^{-1}(\psi),\\
\psi'+(\underline{\delta}+\jj '(\underline{Q}))\psi&= h_1(\underline{\Lambda})\jj '(\underline{Q})R_\tau S_{\underline{\Lambda}}^{-1}(\psi).
\end{align*}
We shall apply $S_{\underline{\Lambda}}R_{\tau}^{-1}$ at both sides of the last equality. Because $(R_{\tau}^{-1}\psi)(t)=\psi(t+\tau)$, we obtain
\begin{align*}
S_{\underline{\Lambda}}R_{\tau}^{-1}\big(\psi'+(\underline{\delta}+\jj  '(\underline{Q}))\psi\big)&=h_1(\underline{\Lambda})\jj  '(\underline{Q})\psi,\\    
S_{\underline{\Lambda}}\big(\psi'(t+\tau)+(\underline{\delta}+\jj  '(\underline{Q}))\psi(t+\tau)\big)&=h_1(\underline{\Lambda})\jj  '(\underline{Q})\psi.   
\end{align*}
Expanding the definitions of $S_{\underline{\Lambda}}$, we get an expression in terms of $\psi(t)$, $\psi'(t)$, $\psi(t+\tau)$ and $\psi'(t+\tau)$. The resulting equation is of the form
\begin{equation}\label{ecuacion general}
A\psi'(t+\tau)+B\psi(t+\tau)=a\psi'(t)+b\psi(t),
\end{equation}
with
\begin{equation}\label{constantes}
    \begin{array}{ll}
    A=1,                &B=\underline{\delta}+ \jj '(\underline{Q}),\\ 
    a=h_{2}(\underline{\Lambda}),  &b=2e^{-\underline{\gamma} \tau}( \underline{K}\underline{\delta}+\jj' (\underline{Q})).\end{array}
\end{equation}
From now  on,  the arguments $\underline{Q}$ and $\underline{\Lambda}$ shall be omitted to simplify notations. In summary, the kernel of $(D_{Q}\F)(\underline{Q},\underline{\Lambda})$ is non-trivial if and only if the equation \eqref{ecuacion general} has a non-trivial solution in $C_T$. 
Let us take a generic function $\psi \in C_T$, and expand it in complex Fourier series, with $\omega=\frac{1}{T}$. We get
\begin{align}
    \psi(t)&=\sum\limits_{k\in\mathbb{Z}}a_ke^{ik\omega t}, \qquad &\psi'(t)&=\sum\limits_{k\in\mathbb{Z}}a_k ik\omega e^{ik\omega t},\\ \psi(t-\tau)&=\sum\limits_{k\in\mathbb{Z}}a_k e^{ik\omega \tau}e^{ik\omega t}, \qquad &\psi'(t-\tau)&=\sum\limits_{k\in\mathbb{Z}}a_k ik\omega e^{ik\omega \tau}e^{ik\omega t}.
\end{align}
By replacing in the equation \eqref{ecuacion general} and comparing coefficients, we obtain
\begin{equation}
a_k\big(e^{ik\omega\tau}(Aik\omega+B)-(aik\omega+b)\big)=0.
\end{equation}
In order to have a non-trivial periodic solution, we need that at least for some $k\in\mathbb{Z}$, the following identity is satisfied: 
\begin{equation}\label{ecuacion caracteristica}
e^{ik\omega\tau}=\frac{aik\omega+b}{Aik\omega+B}.
\end{equation}
Let's call this equation the characteristic equation. For fixed values of $\underline{\Lambda}$, $\tau$ and $T$, this equation may or may not have integer solutions $k$.

\begin{lemma}
For fixed $\underline{\Lambda}$ and $T$, there exists a set $E\subset \R$ such that for $\tau \in \R \setminus E$ the equation \eqref{ecuacion caracteristica} has no integer solutions. $E$ is empty for almost all values of $\underline{\Lambda}$ and $T$, and countable for the remaining ones.
\end{lemma}

\begin{proof}
Consider the homography $\mathcal{H}(z)=\frac{aiz+b}{Aiz+B}$. The image of the real line under $\mathcal{H}$ is either a circle or a straight line in $\mathbb{C}$. In order to decide which is the case, it suffices to compute the value of the function at three points on the real line. 
\[
\begin{array}{ccccc}
  \mathcal{H}(0)=\dfrac{b}{B}, & \mathcal{H}(1)=\dfrac{ai+b}{Ai+B},& \mathcal{H}(-1)=\dfrac{-ai+b}{-Ai+B} & \text{and} &\mathcal{H}(\infty)=\dfrac{a}{A}.
\end{array}
\]
Thus, $\mathcal{H}(\R)$ is a circle  centered on the real axis and intersecting this  axis at $\mathcal{H}(0)$ and $\mathcal{H}(\infty)$. Hence $\mathcal{H}(\R_{>0})$ is a semicircle and the possible scenarios are the following:
\begin{enumerate}
\item If  $\mathcal{H}(\R_{>0})\cap S^1=\emptyset$, then the equation \eqref{ecuacion caracteristica} has no solutions, for any $\tau\in\R$. Hence, $E_T=\emptyset$. 
\item If $\mathcal{H}(\R_{>0})\cap S^1\neq \emptyset$, then there exists $r\in\R_{>0}$ and $\eta\in\R$ such that 
\begin{equation*}
e^{i\eta}=\mathcal{H}(r).
\end{equation*}
\begin{enumerate}
\item If $T$ is such that $r=k_0\omega$ for some $k_0\in\mathbb{Z}$, then the equation \eqref{ecuacion caracteristica} has solutions for $\tau = \frac{\eta + 2l\pi}{k_0\omega}$, with $l\in \mathbb{Z}$. That is: $E_T=\lbrace \frac{\eta + 2l\pi}{k_0\omega}$ $\mid l\in \mathbb{Z}\rbrace$.
\item If  $r \neq k \omega$ for all $k\in\mathbb{Z}$, then the equation \eqref{ecuacion caracteristica} has no solution independently of $\tau$. Hence, $E_T=\emptyset$.
\end{enumerate}
\end{enumerate}
\end{proof}
\begin{remark}
Given that $\mathcal{H}(\infty)=h_2<1$ (hypothesis \ref{H1}), a more detailed analysis of the value of $\mathcal{H}(0)$ indicates that option 1. is impossible under assumptions \ref{H0}-\ref{H3}.
\end{remark}
 \begin{theorem}
Assume \ref{H1} holds for some constant $\underline{\Lambda}$. Then, for each period $T$ (except at most countably many), there exist open $U, V$ with $\underline{\Lambda} \in U\subset \C^3$ and $V\subset \C$ and a continuous map $I: U \to V$ such that $I(\Lambda)$ is a periodic solution to the system \eqref{Eq1}-\eqref{Eq2} with parameter $\Lambda$. Moreover, $I(\Lambda)$ is unique in $V$.

\end{theorem}

\section{Local stability of the trivial solution}

In this section, we shall prove that if the condition \ref{H3} of Theorem \ref{teoremappal} is replaced by the condition
\begin{description}
\item [\Hy{(H3')}\label{H3'}] $\delta(t)>\beta_0\alpha,$ for all $t\in\R$,
\end{description}
then the solutions of the nonautonomous system \eqref{Eq1}-\eqref{Eq2}, with small positive initial conditions, are bounded from above by the solutions of the autonomous system
\begin{empheq}[left=\empheqlbrace]{align}
Q'(t) &=-(\delta^{*}+\beta(Q(t))Q(t)+h_1^{*}u(t-\tau)\label{Eq1*},\\
u(t) &=\beta(Q(t))Q(t)+h_2^{*}u(t-\tau)\label{Eq2*},
\end{empheq}
with the same initial conditions, and \[\delta^{*}=\min (\delta)-\epsilon, \quad h_{i}^{*}=\max (h_i) +\epsilon, \quad \text{for some} \; \epsilon \ll 1.\] 
This, in turn, implies the local stability of the trivial solution because, as we shall see, the system \eqref{Eq1*}-\eqref{Eq2*} is globally asymptotically stable at the origin. The proof of stability for the autonomous system (taken from \cite{AdiCheTouDCDSB2015}), is based on a Lyapunov  functional argument. An adaptation of this argument for the nonautonomous system seems to be elusive. For this reason, we shall employ a different approach, which consists in using the solutions of the autonomous problem as bounds for the nonautonomous one. 

First we recall some definitions.
\definition{Consider a system described by the coupled differential-functional equations
\begin{empheq}[left=\empheqlbrace]{align}
Q'(t) &=f(t,Q(t),u_t),\label{SistGeneral1}\\
u(t) &=g(t,Q(t),u_t).\label{SistGeneral2}
\end{empheq}
The function g or the subsystem \eqref{SistGeneral2} defined by $g$ is
said to be uniformly input to state stable if there exist:
\begin{enumerate}
\item
A function
$\xi:\R_{\geq 0} \times \R_{\geq 0} \to \R_{\geq 0}$  such that 
$\xi (a, t )$ is continuous, strictly increasing
with respect to $a$, strictly decreasing with respect to $t$, $\xi (0 , t ) = 0$,
and $\lim_{t \to \infty} \xi (a, t ) = 0$.
\item A function $\nu : \R_{\geq 0} \to \R_{\geq 0}$ 
continuous, strictly increasing, with $\nu( 0 ) = 0$, 
\end{enumerate}
such that the solution $u_t( t_0 , \varphi, Q )$ corresponding to the initial condition $u_{t_0} = \varphi$
and input function $Q ( t )$ satisfies
\[\norm {u_t ( t_0 , \varphi, Q )} \leq \xi(\norm{ \varphi }, t -t_0 ) + \nu (\norm{\restr{Q}{[ t_0 , t )} }).\]
}
\begin{theorem} Suppose that $f$ and $g$ map $\R \times (\mbox{bounded sets in } \R^m \times
C[0,1] )$ into bounded sets of  $\R^m$ and $\R^n$ respectively, and $g$ is uniformly
input to state stable; $v_1 , v_2, w : \R_{\geq 0} \to \R_{\geq 0}$ are continuous
nondecreasing functions, where additionally $v_i ( s )$ are positive
for $s > 0$, and $v_i ( 0 ) =  0$. If there exists a functional
\[V : \R \times \R^m \times C[0,1] \to \R,\] such that
\[v_1 (|Q|) \leq V ( t , Q, \varphi)\leq v_2(\norm{(Q, \varphi)})\]
and
\[\dot{V} ( s , Q(s), u_s):=\restr{\dfrac{d}{dt}V(t , Q(t), u_t)}{t=s} \leq-w(|Q(s)|),\]
then, the trivial solution of the coupled differential-functional
equations \eqref{SistGeneral1}-\eqref{SistGeneral2} is uniformly stable. If $w( s )>0$ for $s > 0$, then it
is uniformly asymptotically stable. If, in addition, $\lim_ {s\to \infty} v_1 ( s ) = \infty$ ,
then it is globally uniformly asymptotically stable.
\end{theorem}
\begin{proof}
The proof can be found in \cite{GuLiuA2009}.
\end{proof}
\begin{theorem}
If the conditions \ref{H0},\ref{H1} and \ref{H3'} hold, then the autonomous system \eqref{Eq1*}-\eqref{Eq2*} is globally asymptotically stable.
\end{theorem}
\begin{proof}
For $t\in [0,\tau]$, we have
\[u(t)\leq C \norm{\restr{Q}{[0,t)}} + h_2^{*}\varphi(t-\tau).\]
If we define $n(t):= \min\lbrace n\in \mathbb{N}: n>\frac{t}{\tau}\rbrace$ then by induction
\[u(t)\leq C \left( \dfrac{1-(h_2^{*})^{n(t)}}{1-h_2^{*}} \right)\norm{\restr{Q}{[0,t)}}+(h_2^{*})^{n(t)}\norm{\varphi}.\]
In consequence,
\[u(t)\leq C \left( \dfrac{1}{1-h_2^{*}} \right)\norm{\restr{Q}{[0,t)}}+(h_2^{*})^{t / \tau}\norm{\varphi}.\]
This implies that $\xi (a,t)=(h_2^{*})^{t /\tau}a$ and $\nu(a)=C \left( \dfrac{1}{1-h_2^{*}} \right)a$ satisfy the conditions for uniformly input to state stability.

Next, define \[V(t,Q,\varphi):=|Q|+\dfrac{h_1^{*}}{1-h_2^{*}}\int_{-\tau}^{0}|\varphi(\theta)|d\theta.\]
It is immediate to verify that 
\begin{align} \label{lyapunov}
    |Q|\leq V(t,Q,\varphi)\leq |Q|+\dfrac{h_1^{*}}{1-h_2^{*}}\tau \norm{\varphi} \leq \left( 1 + \dfrac{h_1^{*}}{1-h_2^{*}}\tau\right)\norm{(Q,\varphi)}
\end{align}
and, over trajectories of positive solutions,
\label{derivada}
\begin{align*}
    &\dot{V} ( t , Q(t), u_t)=Q'(t)+ \dfrac{h_1^{*}}{1-h_2^{*}} (u(t)-u(t-\tau)),\\
                        &=-\delta^{*}Q(t)-\beta(Q)Q+h_1^{*}u(t-\tau)+\dfrac{h_1^{*}}{1-h_2^{*}}\left(\beta (Q)Q+h_2^{*}u(t-\tau)-u(t-\tau)\right),\\
                        &=-\delta^{*}Q(t)+\beta(Q)Q(Q)\left(\dfrac{h_1^{*}}{1-h_2^{*}}-1\right)+u(t-\tau)\left(\dfrac{h_1^{*}h_2^{*}}{1-h_2^{*}}+h_1^{*}-\dfrac{h_1^{*}}{1-h_2^{*}}\right),\\
                        &=-\left(\delta^{*}-\beta(Q)\left(\dfrac{h_1^{*}}{1-h_2^{*}}-1\right)\right)Q. 
\end{align*}
   This implies that $V$ is a Lyapunov functional for the system, with $v_1(s)=s$, $v_2(s)=\left( 1 + \dfrac{h_1^{*}}{1-h_2^{*}}\tau\right)s$ and $w(s)=\delta^{*}-\beta(s)\left(\dfrac{h_1^{*}}{1-h_2^{*}}-1\right)>0$. We remark that the quantities  $\delta^{*}$ and $h_i^{*}$ were defined in order to guarantee that the latter inequality is  strict. 
\end{proof}
\begin{lemma}
Let $\bar{r}>0$ be the value where $\jj$ reaches its maximum.  If $\norm{(Q_0,\varphi)}$ is small enough then $Q(t)<\bar{r}$ for all $t\geq 0$.
\end{lemma}
\begin{proof}
Take 
\[\norm{(Q_0,\varphi)}\leq \dfrac{\bar{r}}{1 + \dfrac{h_1^{*}}{1-h_2^{*}}\tau}. \]
Then, by \eqref{lyapunov} we have \[V(0,Q_0,\varphi)\leq  \bar{r}.\] 
Also, because of \eqref{derivada}, 
\[\dot{V}(t,Q(t),u_t)\leq  0, \quad \text{for all} \; t\geq 0,\]
and again by \eqref{lyapunov}, this implies $|Q(t)|\leq \bar{r}$, for all $t\geq 0$.
\end{proof}
Remark that for the Hill function $\beta(Q):=\dfrac{\beta_0}{1+Q^r}$, $\beta_0>0$,  $r>1$, we have 
\[\bar{r}=\left( \dfrac{1}{r-1}\right)^{\frac{1}{r}}.\]

The following result shall provide a comparison between the solutions $(Q,u)$ and $(\Q,\u)$, the solutions to the nonautonomous and the autonomous case respectively, for 
given initial conditions $(Q_0,\varphi)$.

\begin{theorem}
Assume the initial conditions satisfy \[\norm{(Q_0,\varphi)}\leq \dfrac{\bar{r}}{1 + \dfrac{h_1^{*}}{1-h_2^{*}}\tau}.\]
Then, $Q(t)\leq \Q(t)$ and $u(t)\leq\u(t)$, for all $t\geq 0$.
\end{theorem}

\begin{proof}
The proof will proceed by the method of steps. Let $t\in [0,\tau]$, then
\begin{align*}
(Q-\Q)'(t)&\leq -\delta(t)Q(t)-\jj(Q(t))+\delta^{*}\Q(t)+\jj(\Q(t))+h_1(t)u(t-\tau)-h_1^{*}\u(t-\tau),\\
&< -\delta^{*}(Q-\Q)(t)-(\jj(Q(t))-\jj(\Q(t)))+h_1^{*}(u-\u)(t-\tau).
\end{align*}
Now, because $u(t-\tau)=\u(t-\tau)=\varphi(t-\tau)$, we get 
\[(Q-\Q)'(t) < -\delta^{*}(Q-\Q)(t)-(\jj(Q(t))-\jj(\Q(t))).\]
As $(Q-\Q)(0)=0$ and  $(Q-\Q)'(0)<0$, so $(Q-\Q)$ starts negative. Suppose there exists $t_{0}\in[0,\tau]$ such that $Q(t_{0})=\Q(t_{0})=0$ and $Q(t_{0})<\Q(t_{0})$ for $0\leq t<t_0$. Then, $(Q-\Q)'(t_{0})<0$, which is a contradiction. So, $Q(t)<\Q(t)$ for all $t\in [0,\tau]$. In particular, $Q(t)<\bar{r}$ for all $t\in[0,\tau]$. So, given that $\jj$ is increasing in $[0,\bar{r}]$, $\jj (Q)-\jj(\Q)<0.$ For the second equation in $[0,\tau]$, 
\[(u-\u)(t)<\jj (Q)-\jj(\Q)<0.\]
Now, for $t\in[\tau,2\tau]$, $t-\tau \in [0,\tau]$. So, $(u-\u)(t-\tau)<0$ and then
\begin{align*}
(Q-\Q)'(t)&< -\delta^{*}(Q-\Q)(t)-(\jj (Q)-\jj(\Q))+h_1^{*}(u-\u)(t-\tau),\\
&<-\delta^{*}(Q-\Q)(t)-(\jj (Q)-\jj(\Q)).
\end{align*}
Given $(Q-\Q)(\tau)<0$ and $(Q-\Q)'(\tau)<0$, by a similar argument as before,
\[Q(t)<\Q(t)<\bar{r}, \quad \text{for all} \; t\in[\tau,2\tau].\]
Similarly, for the second equation,
\begin{align*}
(u-\u)'(t)&< \jj(Q(t))-\jj(\Q(t))+h_1^{*}(u-\u)(t-\tau)<0.
\end{align*}
The result follows inductively.
\end{proof}

\begin{corollary}
Suppose that 
\[\norm{(Q_0,\varphi)}\leq \dfrac{\bar{r}}{1 + \dfrac{h_1^{*}}{1-h_2^{*}}\tau}.\]
Then, the solutions of the original system tend asymptotically to zero. That is, the trivial solution is locally asymptotically stable.
\end{corollary}


\begin{thebibliography}1

\bibitem{AdiCheTouDCDSB2015}
M. Adimy, A. Chekroun and T.M. Touaoula, Age-structured and delay differential-difference model of hematopoietic stem cell dynamics. Discrete and Continuous Dynamical Systems - Series B, 20 (2015), 2765-2791. 

\bibitem{AdiCraRuaSIAMJAM2005}
M. Adimy, F. Crauste and S. Ruan, A mathematical study of the hematopoiesis process with applications to chronic myelogenous leukemia. SIAM J. Appl. Math., 65 (2005), 1328–1352.

\bibitem{ApoMacJTB2008}
R. Apostu and M.C. Mackey, Understanding cyclical thrombocytopenia: a mathematical modeling approach. J. Theor. Biol., 251 (2008), 297–316.

\bibitem{BerHanGI2006}	
S. Bernard and H. Herzel, Why do cells cycle with a 24 hour period. Genome Informatics. International Conference on Genome Informatics, 17(1) (2006), 72-79.

\bibitem{BjaJorSotAJP1999}
G. A. Bjarnason, R. C. K. Jordan and R. B. Sothern, Circadian variation in the expression of cell-cycle proteins in human oral epithelium. The American journal of pathology. 154 (2) (1999), 613-622.

\bibitem{BroS1993}
R. F. Brown, A Topological Introduction to Nonlinear Analysis, Springer (1993).

\bibitem{BurTanCP1970}
F. J. Burns and I. F. Tannock, On the existence of a G$_0$-phase in the cell cycle, Cell Proliferation, 3 (1970), 321–334.

\bibitem{ClaMicPerCRM2006}
J. Clairambault, P. Michel and B. Perthame, Circadian rhythm and tumour growth. Comptes Rendus Mathematique 342 (1) (2006), 17-22.	

\bibitem{ColMacJTB2005}
C. Colijn and M.C. Mackey, A mathematical model of hematopoiesis – I. Periodic chronic myelogenous leukemia. J. Theor. Biol., 237 (2005), 117–132.

\bibitem{ColMacJTB2005Bis}
C. Colijn and M.C. Mackey, A mathematical model of hematopoiesis – II. Cyclical neutropenia. J. Theor. Biol., 237 (2005), 133–146.

\bibitem{DiaZhoAMC2016}
T. Diagana and H. Zhou, Existence of positive almost periodic solutions to the hematopoiesis model. Appl. Math. Comput. 274 (2016) 644–648.

\bibitem{DinLiuNieAMM2016}
H. S. Ding, Q. L. Liu and J. J. Nieto, Existence of positive almost periodic solutions to a class of hematopoiesis model. Appl. Math. Model. 40 (2016) 3289–3297.

\bibitem{FicMurLinCleCSC2008}
F. Ficara, M. J. Murphy, M. Lin and M. L. Cleary, Pbx1 regulates self-renewal of long-term hematopoietic stem cells by maintaining their quiescence, Cell Stem Cell, 2 (2008), 484–496.

\bibitem{FolMacJMB2009}
C. Foley and M.C. Mackey, Dynamic hematological disease: a review. J. Math. Biol., 58 (2009), 285–322.

\bibitem{ForMacJH1999}
P. Fortin and M.C. Mackey, Periodic chronic myelogenous leukaemia: spectral analysis of blood cell counts and a etiological implications. Br. J. Haematol., 104 (1999), 336–345.

\bibitem{FuPelLiuEtAllC2002}
L. Fu, H. Pelicano, J. Liu, P. Huang and C. C. Lee, The Circadian Gene Period2 Plays an Important Role in Tumor Suppression and DNA Damage Response In Vivo. Cell 111 (2002), 41-50.

\bibitem{GerKomBalEtAllMC2006}
S. Gery, N. Komatsu, L. Baldjyan, A. Yu, D. Koo and H. P. Koeffler, The circadian gene per1 plays an important role in cell growth and DNA damage control in human cancer cells. Molecular Cell. 22(3) (2006), 375-382.

\bibitem{GuLiuA2009} 
K. Gu and Y. Liu, Lyapunov-Krasovskii functional for uniform stability of coupled differential-functional equations, Automatica, 45 (2009), 798–804.

\bibitem{HauDalMacB1998}
C. Haurie, D.C. Dale and M.C. Mackey, Cyclical neutropenia and other periodic hematological disorders: A review of mechanisms and mathematical models. Blood, 92 (1998), 2629–2640.

\bibitem{HauPerDalMacEH1999}
C. Haurie, R. Person, D.C. Dale and M.C. Mackey, Hematopoietic dynamics in grey collies. Exp. Hematol., 27 (1999), 1139–1148.

\bibitem{KapKhuIJDE2019}
P. Kapula and M. Khuddush, Existence and Global Exponential Stability of Positive Almost Periodic Solutions for a Time-Scales Model of Hematopoiesis with Multiple Time-Varying Variable Delays. International Journal of Difference Equations. 14(2) (2019), 149–167.

\bibitem{LeiMacJTB2011}
J. Lei and M. C. Mackey, Multistability in an age-structured model of hematopoiesis: Cyclical neutropenia, Journal of Theoretical Biology, 270 (2011), 143–153.

\bibitem{MacB1978} 
M. C. Mackey, Unified hypothesis for the origin of aplastic anemia and periodic hematopoiesis. Blood, 51 (1978), 941–956.

\bibitem{MacBMB1979}
M.C. Mackey, Periodic auto- immune hemolytic anemia: an induced dynamical disease. Bull. Math. Biol., 41 (1979), 829–834.
	
\bibitem{MatYamMitEtAllS2003}
T. Matsuo, S. Yamaguchi, S. Mitsui, A. Emi, F. Shimoda and H. Okamura, Control mechanism of the circadian clock for timing of cell division in vivo, Science 302 (2003) 255–259.	
	
\bibitem{MilMacJRCP1989}
J.G. Milton and M.C. Mackey, Periodic haematological diseases: mystical entities of dynamical disorders? J.R. Coll. Phys., 23 (1989), 236–241.
	
\bibitem{PotBooCraEtAllCP2002}
C. S. Potten, D. Booth, N. J. Cragg, G. L. Tudor, J. J. O’Shea, D. Appleton, D. Barthel, T. G. Gerike, F. A. Meineke, M. J. Loeffler and C. Booth, Cell kinetic studies in the murine ventral tongue epithelium: thymidine metabolism studies and circadian rhythm determination. Cell Proliferation. 35 (1) (2002), 1-15.

\bibitem{PujBerMacJADS2005}
L. Pujo-Menjouet, S. Bernard and M.C. Mackey, Long period oscillations in a G0 model of hematopoietic stem cells. SIAM J. Appl. Dyn. Systems, 4 (2005), No. 2, 312–332.

\bibitem{PujMacCRB2004}
L. Pujo-Menjouet and M.C. Mackey, Contribution to the study of periodic chronic myelogenous leukemia. Comptes Rendus Biologies, 327 (2004), 235–244.	
	
\bibitem{SmaLaeLotEtAllB1991}
R. Smaaland, O. D. Laerum, K. Lote, O. Sletvold, R. Sothern and R. Bjerknes, DNA synthesis in human bone marrow is circadian stage dependent. Blood, 77 (1991), 2603-2611.

\bibitem{SanBelMahMacJTB2000}
M. Santillan, J. B\'elair, J.M. Mahaffy and M.C. Mackey, Regulation of platelet production: The normal response to perturbation and cyclical platelet disease. J. Theor. Biol., 206 (2000), 585–603.
	
\bibitem{VegWinMelIL2010}
P. Vegh, J. Winckler and F. Melchers, Long-term ”in vitro” proliferating mouse hematopoietic progenitor cell lines, Immunology Letters, 130 (2010), 32–35.

\bibitem{WenCMA2002}
P. X. Weng, Global attractivity of periodic solution in a model of hematopoiesis. Comput. Math. Appl. 44 (2002) 1019–1030.

\bibitem{WilLauOseEtAllCP2008}
A. Wilson, E. Laurenti, G. Oser, R. C. van der Wath, W. Blanco-Bose, M. Jaworski, S. Offner, C. F. Dunant, L. Eshkind, E. Bockamp, P. Li\'o, H. R. MacDonald and A. Trumpp, Hematopoietic stem cells reversibly switch from dormancy to self-renewal during homeostasis and repair, Cell, 135 (2008), 1118–1129.

\bibitem{XuLiADE1998}
W. Xu and J. Li, Global attractivity of the model for the survival of red blood cells with several delays. Ann. Differential Equations 14 (1998) 357–363.

\bibitem{YaoJNSA2015}
Z. Yao, Uniqueness and global exponential stability of almost periodic solution for Hematopoiesis model on time scales. J. Nonlinear. Sci. Appl. 8 (2015) 142–152.

\bibitem{ZhaYanWanAML2013}
H. Zhang, M. Q. Yang and L. J. Wang, Existence and exponential convergence of the positive almost periodic solution for a model of hematopoiesis. Appl. Math. Lett. 26 (2013) 38–42.

\bibitem{ZhoWanZhoAAA2013}
H. Zhou, W. Wang and Z. F. Zhou, Positive almost periodic solution for a model of hematopoiesis with infinite time delays and a nonlinear harvesting term. Abstr. Appl. Anal. (2013) ID 146729, 6 p.

\bibitem{ZhoYanJMAA2018}
H. Zhou, L. Yang, A new result on the existence of positive almost periodic solution for generalized hematopoiesis model. J. Math. Anal. Appl. 462 (2018) 370– 379.


\end{thebibliography}
\end{document}